\documentclass[11pt]{amsart}
\usepackage{amsmath}
\usepackage{amsthm}
\usepackage{amsfonts}
\usepackage{amssymb}
\usepackage{hyperref}

\theoremstyle{plain}
\newtheorem{theorem}{Theorem}
\newtheorem{lemma}[theorem]{Lemma}

\numberwithin{equation}{section}

\let\abs=\envert
\let\vph=\varphi

\newcommand{\mathmod}[1]{\ \left(\mathrm{mod}\ #1\right)}

\newcommand{\rad}{\mathrm{rad}}

\begin{document}
\title{On Lehmer's problem and related problems}
\author{Tomohiro Yamada}
\address{Institute for Promotion of Higher Education, Kobe University,
657-0011, 1-2-1, Tsurukabuto, Nada, Kobe, Hyogo, Japan}
\email{tyamada1093@gmail.com}

\date{}

\begin{abstract}
We show that if $N\pm 1=M\vph(N)$ with $N\neq 15, 255$ composite,
then $M<15.76515\log\log\log N$ and $M<16.03235\log\log\omega(N)$,
together with similar results for the unitary totient function, Dedekind function, and the sum of unitary divisors.
\end{abstract}

\maketitle

\section{Introduction}\label{intro}

Let $\vph(N)$ denote the Euler totient function of $N$.
Clearly, $\vph(p)=p-1$ for any prime $p$.
Lehmer~\cite{Leh} conjectured that there exists no composite number $N$ such that
$\vph(N)$ divides $N-1$ and showed that such an integer must be an odd squarefree integer
with at least seven prime factors.
In other words, if $\vph(N)\mid (N-1)$ and $N$ is composite, then $N$ is odd and
$\omega(N)=\Omega(N)\geq 7$, where $\omega(N)$ and $\Omega(N)$ respectively
denote the number of distinct and not necessarily distinct prime factors of $N$.

For such an integer $N$, Cohen and Hagis~\cite{CH} showed that $\omega(N)\geq 14$ and $N>10^{20}$,
Renze's notebook \cite{Ren} shows that $\omega(N)\geq 15$ and $N>10^{26}$,
and Pinch claims that $N>10^{30}$ at his research page \cite{Pin}.
Burcsi, Czirbusz, and Farkas~\cite{BCF} proved that if $3\mid N$, then $\omega(N)\geq 4\times 10^7$ and $N>10^{3.6\times 10^8}$.
Pomerance~\cite{Pom} showed that the number of such integers $N\leq x$ is $O(x^{1/2}\log^{3/4} x)$
and $N\leq r^{2^r}$ if $2\leq \omega(N)\leq r$ additionally.
Luca and Pomerance~\cite{LP} showed that the number of such integers $N\leq x$ is at most
$x^{1/2}/\log^{1/2+o(1)} x$.
Furthermore,
Burek and \.{Z}mija~\cite{BZ} showed that $N\leq 2^{2^r}-2^{2^{r-1}}$ if $\vph(n)$ divides $N-1$
and $2\leq \omega(N)\leq r$.

For integers $N$ such that $N-1=M\varphi(N)$ with $M$ a large integer, stronger results are known.
Hagis proved that if $N-1=3\varphi(N)$, then $\omega(N)\geq 1991$ and $M>10^{8171}$.
For integers $N=M\varphi(N)+1$, $M\geq 4$,
Grytczuk and W\'{o}jtowicz~\cite{GW} showed that
$\omega(N)\geq 3049^{M/4}-1509$ if $3\mid N$
and $\omega(N)\geq 143^{M/4}-1$ otherwise.

Subbarao~\cite{Sub} considered the problem analogous to Lehmer's problem
involving $\vph^*$, the unitary analogue of $\vph$.
So $\vph^*$ is defined by
\begin{equation}
\vph^*(N)=\prod_{p^e\mid\mid N}(p^e-1),
\end{equation}
where the product is over all prime powers unitarily dividing $N$.
We call the value $\vph^*(N)$ the {\it unitary totient} of an integer $N$.
Subbarao conjectured that $\vph^*(N)$ divides $N-1$ if and only if $N$ is a prime power.
This conjecture is still unsolved.
However, Subbarao and Siva Rama Prasad~\cite{SS}
showed that $N$ must have at least eleven distinct prime factors if $N$ is not a prime power
and $\vph^*(N)$ divides $N-1$.
Moreover, Siva Rama Prasad, Goverdhan, and Al-Aidroos~\cite{PGA} proved that for integers
$N=M\varphi^*(N)+1$ with $M\geq 4$,
\vspace{-0.6\baselineskip}{\setlength{\leftmargini}{18pt}
\begin{itemize}
\item[1.] $\omega(N)>(800000)^{M/4}-499883$ and $N>(k_1 M \beta_1^M)^{\beta_1^M}$ if $15\mid N$,
\item[2.] $\omega(N)>(597515)^{M/4}-298668$ and $N>(k_2 M \beta_2^M)^{\beta_2^M}$ if $3\mid N, 5\nmid N$,
\item[3.] $\omega(N)>(1889)^{M/4}-468$ and $N>(k_3 M \beta_3^M)^{\beta_3^M}$ if $3\nmid N, 5\mid N$, and
\item[4.] $\omega(N)>(608)^{M/4}-3$ and $N>(k_4 M \beta_4^M)^{\beta_4^M}$ otherwise,
\end{itemize}
}\vspace{-0.5\baselineskip}
where $(\beta_1, \beta_2, \beta_3, \beta_4)=(23.4, 23.38, 6.1, 4.9)$ and $k_j=(\log \beta_j)/3$ for $j=1, 2, 3, 4$.
We prove the following upper bounds for $M$.
\begin{theorem}\label{th11}
Let $N_1$ denote the product of prime factors $p$ dividing $N$ exactly once here and hereafter.
If $M\vph^*(N)=N\pm 1$, then $M<19.44947\log\log\log N_1$ for $N_1\geq 23$ or $N_1=19$.
Moreover, if $M\vph(N)=N\pm 1$, then $M<15.76515\log\log\log N$ for $N\geq 19$.
\end{theorem}
\begin{theorem}\label{th12}
If $M\vph^*(N)=N\pm 1$ and $\omega(N_1)\geq 4$, then $M<19.77911\log\log\omega(N_1)$.
Moreover, if $M\vph(N)=N\pm 1$ and $\omega(N)\geq 4$, then $M<16.03235\log\log\omega(N)$.
\end{theorem}
As Lehmer \cite{Leh} observed, we see that
$M\vph(N)=N\pm 1$ and $\omega(N)\leq 3$, then $N$ must be prime or $N=15, 255$.
Hence, if $M\vph(N)=N\pm 1$ with $N\neq 15, 255$ composite, then $M<15.76515\log\log\log N$
and $M<16.03235\log\log\omega(N)$.

Subbarao~\cite{Sub} also studies similar problems
for Dedekind function $\psi(N)=N\prod_{p^e\mid\mid N}p^{e-1}(p+1)$
and the sum $\sigma^*(N)=\prod_{p^e\mid\mid N}(p^e+1)$ of unitary divisors of $N$.
Clearly, $\sigma^*(N)=N+1$ if and only if $N$ is a prime power.
Moreover, if $\psi(N)=aN+b$ and $\gcd(b, N)=1$ with $a, b$ integers,
then $N$ must be squarefree and $\sigma^*(N)=\psi(N)=aN+b$.
Subbarao proved that if $\sigma^*(N)=MN+1$ with $M>1$ and $\omega(N)=r$,
then $M\geq 3$ must be odd, $r\geq 16$, and $10^{20}<N<(r-1)^{2^{r-1}}$.
Subbarao also proved that if $\psi(N)=MN+1$ with $M>1$ and $3\mid N$, then $\omega(N)\geq 185$.
Hasanalizade~\cite{Has} proved that if $\sigma^*(N)=MN+1$ with $M>1$, then
$N>((\log 2)(AM^2-1)2^{AM^2-1}/3)^{2^{AM^2-1}}$ and $\omega(N)>1578^{AM^2/9}/2$, where
$A=0.998\cdots$, if $3\mid N$ and
$N>((\log 3)M3^{M-1})^{3^M}$ and $\omega(N)>51^{M/3}-1$ if $3\not\mid N$.
We prove the following upper bounds for $M$.
\begin{theorem}\label{th13}
If $\sigma^*(N)=MN\pm 1$, then $M<18.87067\log\log\log N_1$ for $N_1\geq 19$.
Moreover, if $\psi(N)=MN\pm 1$, then $M<15.52051\log\log\log N$ for $N\geq 19$.
\end{theorem}
\begin{theorem}\label{th14}
If $\sigma^*(N)=MN\pm 1$ and $\omega(N_1)\geq 4$, then $M<19.40333\log\log\omega(N_1)$.
Moreover, if $\psi(N)=MN\pm 1$ and $\omega(N)\geq 4$, then $M<15.72775\log\log\omega(N)$.
\end{theorem}

Our upper bounds are eventually stronger than known bounds in the sense of being
at least of triple-exponential and double-exponential order of $M$ for $N$ and $\omega(N)$ respectively.

\section{Explicit sieve estimates}

We write the summatory function of an arithmetic function $f$ for $M_f(x)=\sum_{n\leq x}f(n)$.
For a set $U$ of primes, we put
\begin{equation*}
P_U(x)=\prod_{p\in U, p\leq x}\left(1-\frac{1}{p}\right)^{-1},
S_U(x)=\sum_{p\in U, p\leq x}\frac{1}{p},
\theta_U(x)=\sum_{p\in U, p\leq x}\log p,
\end{equation*}
and $\pi_U(x)=\sum_{p\in U, p\leq x}1$ to be the number of primes in $U$ below $x$.

Given an integer $a$, we call a set $U$ of primes $a$-\textit{self-repulsive} if
for any two primes $p$ and $q$ in $U$, we have $q\not\equiv a\mathmod{p}$.
Studies of $1$-self-repulsive sets of primes have been begun by Golomb~\cite{Gol},
who observed that if $N$ is an integer such that $\gcd(N, \vph(N))=1$ and
$U$ be the set of prime factors of $N$, then, $U$ must be $1$-self-repulsive.
Indeed, we can easily see that if $\gcd(N, \varphi^*(N))=1$ and
$U$ be the set of prime factors of $N$, then, $U$ must be $1$-self-repulsive.

More generally, letting $\vph_a(N)=\prod_{p^e\mid\mid N}(p-a) p^{e-1}$,
we can easily see that if $\gcd(N, \vph_a(N))=1$, then $N$ is squarefree, $\gcd(N, a)=1$, and
the set of prime factors of $N$ must be $a$-self-repulsive.

Using Brun-Selberg upper bound sieve,
Meijer~\cite{Mei}, who used the term $G$-sequence to mean $1$-self-repulsive set,
proved that there exist some absolute constants $c_1$ and $c_2$ such that,
if $U$ is a $1$-self-repulsive set of primes, then
\begin{equation}
\pi_U(x) P_U(x)\leq \frac{c_1 x}{\log x}
\end{equation}
and
\begin{equation}
P_U(x)\leq c_2 \log\log x
\end{equation}
for $x\geq 3$.

Our purpose of this section is to prove the following explicit estimate for $\pm$-self-repulsive sets.
\begin{theorem}\label{th21}
Let $U$ be an $\pm 1$-self-repulsive set of primes.
Then, for $x>e^{73}$, we have
\begin{equation}
\pi_U(x)
<\frac{8e^\gamma x\left(1+\frac{1}{\log x}\right)\left(1+\frac{1}{2\log^3 x}\right)}
{P_U(x)\log x\left(1-\frac{\log\log x-8\gamma}{\log x}\right)^2\left(1-\frac{\log\log x}{\log x}\right)}.
\end{equation}
\end{theorem}

Instead of Brun-Selberg sieve, we use the large sieve method as in \cite{FNO}, \cite{Ymd19}, and \cite{Ymd20}.
We write $F=G+O^*(H)$ to mean that $\abs{F-G}\leq H$.
Let $x$ be a positive number
and $A$ be a set of integers contained in an interval of length at most $x$.
For each prime $p$, let $\Omega_p$ be a set of residue classes modulo $p$
and $\rho(p)$ denote the number of residue classes in $\Omega_p$.
We consider the number $Z(A, w, \Omega)$ of integers in $A$
that do not belong to $\Omega_p$ for any prime $p\leq w$.
Hence, if $U$ is self-repulsive, then we take $\Omega_p=\{0, 1\mathmod{p}\}$ for primes $p$ in $U$,
$\Omega_p=\{0\mathmod{p}\}$ for primes $p$ outside $U$,
and $A$ to be the set of positive integers below $x$ to obtain
\begin{equation}\label{eq21}
\pi_U(x)\leq Z(A, w, \Omega)+w
\end{equation}
for any real $w$.

Let $g(m)$ be the multiplicative function supported only on the squarefree integers $m$
defined by $g(p)=\rho(p)/(p-\rho(p))$ for each prime $p$ and
\[M_g(z)=\sum_{n\leq z}g(n).\]
Then, as mentioned in the Introduction,
Theorem 7.14 of \cite{IK} immediately gives the following estimate:
\begin{lemma}\label{lm0}
Assume that $\rho(p)<p$ for any prime $p$.
Then, for any $w\geq 1$ we have
\begin{equation}
Z(A, w, \Omega)\leq \frac{x+w^2}{M_g(w)}.
\end{equation}
\end{lemma}

So that, our concern is to obtain a lower estimate for $M_g(x)$ with
$\rho(n)=\rho_U(n)$ the multiplicative function supported on squarefree integers defined by
$\rho(p)=2$ for primes $p$ in $U$ and $\rho(p)=1$ for primes $p$ outside $U$.
Our argument is based on the solution of Exercise 1.27 of \cite{Par}.

We begin by the following general inequality for nonnegative multiplicative functions.
\begin{lemma}\label{lm21}
For a multiplicative function $f(n)$ over positive integers, let
$M_{f, U}(x)=\sum_{n\leq x, \gcd(n, U)=1}f(n)$.
In particular, we have $M_f(x)=M_{f, 1}(x)=\sum_{n\leq x}f(n)$.
If $f(n)$ always takes nonnegative value, then
\begin{equation}
M_{f, U}(x)\geq \frac{M_f(x)}{\prod_{p\in U}\sum_{e\geq 0}f(p^e)}.
\end{equation}
\end{lemma}

\begin{proof}
Let $U_0$ be the set of primes in $U$ below $x$.
We prove the lemma by induction of the number of primes in $U_0$.

If $U_0$ is empty, then clearly we have
\begin{equation}
M_{f, U}(x)=M_{f, U_0}(x)=M_f(x)\geq \frac{M_f(x)}{\prod_{p\in U}\sum_{e\geq 0}f(p^e)}.
\end{equation}

Assume that $U_0$ is nonempty and the lemma holds for $U_1=U_0\setminus\{p\}$ and some prime $p$ in $U_0$.
Then we have
\begin{equation}
M_{f, U}(x)=M_{f, U_0}(x)=\sum_{n\leq x, \gcd(n, U_1)=1}f(n)-\sum_{n\leq x, \gcd(n, U_1)=1, p\mid n}f(n).
\end{equation}

Since
\begin{equation}
\sum_{\substack{n\leq x, \\ \gcd(n, U_1)=1, \\ p\mid n}}f(n)
=\sum_{e\geq 1}\sum_{\substack{m\leq x/p^e, \\ \gcd(m, U_0)=1}}f(p^e)f(m)
\leq M_{f, U_0}(x)\sum_{e\geq 1}f(p^e),
\end{equation}
we have
\begin{equation}
M_{f, U_0}(x)\geq M_{f, U_1}(x)-M_{f, U_0}(x)\sum_{e\geq 1}f(p^e),
\end{equation}
that is,
\begin{equation}
M_{f, U_0}(x)\geq \frac{M_{f, U_1}(x)}{\sum_{e\geq 0}f(p^e)}.
\end{equation}

From the assumption, we have
\begin{equation}
M_{f, U}(x)=M_{f, U_0}(x)\geq\frac{M_f(x)}{\prod_{p\in U}\sum_{e\geq 0}f(p^e)}.
\end{equation}
Now the lemma follows by induction.
\end{proof}

\begin{lemma}\label{lm22}
For $y\geq 60$,
\begin{equation}
\sum_{m\leq y}\frac{\tau(y)}{y}>\frac{\log^2 y}{2}+2\gamma\log y+0.4.
\end{equation}
\end{lemma}

\begin{proof}
Put $D(w)=\sum_{n\leq w}\tau(n)$.
Theorem 1.2 of \cite{BBR} gives that for all $w\geq 9995$,
\begin{equation}
\sum_{n\leq w}\tau(n)=w\log w+(2\gamma-1)w+\Delta(w)
\end{equation}
with $\abs{\Delta(w)}\leq 0.764 w^{1/3} \log w$.
Partial summation gives that
\begin{equation}
\begin{split}
\sum_{n\leq z}\frac{\tau(n)}{n}= & ~ \frac{D(z)}{z}+\int_1^z\frac{D(t)}{t^2} dt \\
= & ~ \frac{\log^2 z}{2}+2\gamma\log z+(2\gamma-1)+\frac{\Delta(z)}{z}+\int_1^z \frac{\Delta(t)}{t^2} dt.
\end{split}
\end{equation}

Taking $B_0=2\gamma-1+\int_1^\infty \Delta(t)t^{-2}dt$, we have
\begin{equation}
\begin{split}
\sum_{n\leq z}\frac{\tau(n)}{n}
= & ~ \frac{\log^2 z}{2}+2\gamma\log z+B_0+\frac{\Delta(z)}{z}+\int_z^\infty \frac{\Delta(t)}{t^2}dt \\
= & ~ \frac{\log^2 z}{2}+2\gamma\log z+B_0+O^*\left(\frac{0.191(10\log z+9)}{z^{2/3}}\right)
\end{split}
\end{equation}
for $z\geq 9995$.
Lemma 1 of \cite{RV} had proved that $B_0=\gamma^2-2\gamma_1=0.478809\cdots$ with the error term $O^*(1.641z^{-1/3})$,
where $\gamma_1=-0.072815\cdots$ is the first Stieltjes constant.
We note that in Corollary 2.2 of \cite{BBR} and Lemma 3.3 of \cite{Ram1}, the constant term $B_0$ is
erroneously given as $\gamma^2-\gamma_1$,
which should be $\gamma^2-2\gamma_1$ as in \cite{RV}.
Now calculation gives the lemma.
\end{proof}

Now we would like to show the following lower bound for $M_g(y)$.
\begin{lemma}
For $y>e^{30}$, we have
\begin{equation}
M_g(y)>P_U(y)e^{-\gamma}\left(\frac{\log y}{2}+2\gamma+\frac{0.1}{\log y}\right).
\end{equation}
\end{lemma}

\begin{proof}
We put $\omega_U(n)$ be the number of distinct prime factors in $U$ of $n$,
$\tau_U(n)$ be the number of divisors in $U$ of $n$,
and $\rad(n)=\prod_{p\mid n} p$ be the product of distinct prime divisors of $n$.

We put $V$ to be the set of integers composed only of primes in $U$.
Then, we see that
\begin{equation}
\begin{split}
\sum_{n\leq y}g(n)
= & ~ \sum_{n\leq y}\prod_{p\mid n, p\in U}\frac{2}{p-2}\prod_{p\mid n, p\not\in U}\frac{1}{p-1} \\
= & ~ \sum_{n\leq y}\mu^2(n)\prod_{p\mid n, p\in U}\sum_{e\geq 1}\left(\frac{2}{p}\right)^e\prod_{p\mid n, p\not\in U}\sum_{e\geq 1}\frac{1}{p^e} \\
= & ~ \sum_{n\leq y}\mu^2(n)\sum_{\rad(k)=n}\frac{2^{\omega_U(k)}}{k}
=\sum_{\rad(k)\leq y}\frac{2^{\omega_U(k)}}{k} \\
\geq & ~ \sum_{k\leq y}\frac{\tau_U(k)}{k}
=\sum_{m\leq y}\left(\frac{1}{m}\sum_{d\leq y/m, d\in V}\frac{1}{d}\right).
\end{split}
\end{equation}
By Lemma \ref{lm21}, we have
\begin{equation}
\begin{split}
\sum_{d\leq y/m, d\in V}\frac{1}{d}
\geq & ~ \prod_{p\leq y/m, p\not\in U}\left(1-\frac{1}{p}\right)\sum_{d\leq y/m}\frac{1}{d} \\
\geq & ~ P_U(y) \prod_{p\leq y}\left(1-\frac{1}{p}\right)\sum_{d\leq y/m}\frac{1}{d}.
\end{split}
\end{equation}
Using Theorem 7 of \cite{RS}, we obtain
\begin{equation}
\sum_{d\leq y/m, d\in V}\frac{1}{d}>\frac{\sum_{d\leq y/m} 1/d}{e^\gamma \log y}\left(1-\frac{1}{2\log^2 y}\right)
\end{equation}
and therefore
\begin{equation}
\begin{split}
\sum_{n\leq y}g(n)> & ~ \frac{P(y)\sum_{m\leq y}\left(\frac{1}{m}\sum_{d\leq y/m}\frac{1}{d}\right)}{e^\gamma \log y}\left(1-\frac{1}{2\log^2 y}\right) \\
= & \frac{P_U(y)\sum_{n\leq y}\frac{\tau(n)}{n}}{e^\gamma \log y}\left(1-\frac{1}{2\log^2 y}\right) \\
> & P_U(y)e^{-\gamma}\left(\frac{\log y}{2}+2\gamma+\frac{0.4}{\log y}\right)\left(1-\frac{1}{2\log^2 y}\right) \\
> & P_U(y)e^{-\gamma}\left(\frac{\log y}{2}+2\gamma+\frac{0.1}{\log y}\right).
\end{split}
\end{equation}
\end{proof}

Now we shall prove Theorem \ref{th21}.
Lemma \ref{lm0} immediately gives
\begin{equation}
Z(A, y, \Omega)\leq \frac{x+y^2}{M_g(y)}
<\frac{e^\gamma (x+y^2)}{P_U(y)\left(\frac{\log y}{2}+2\gamma+\frac{0.12}{\log y}\right)}.
\end{equation}
With the aid of Theorem 5.9 of \cite{Dus}, we have
\begin{equation}
\frac{P_U(x)}{P_U(y)}\leq \prod_{y<p\leq x}\frac{p}{p-1}<\frac{\log x}{\log y}\left(1+\frac{1}{5\log^3 y}\right)^2
\end{equation}
(but Ramar\'{e}'s zero density estimate in \cite{Ram2}, on which Dusart's estimates in \cite{Dus} are based,
is objected by \cite{BKLNW}.
Corollary 11.2 in \cite{BKLNW} can instead be used to obtain Dusart's estimates),
and therefore
\begin{equation}\label{eq22}
Z(A, y, \Omega)<
\frac{e^\gamma (x+y^2)\log x}{P_U(x)(\frac{\log^2 y}{2}+2\gamma\log y+0.12)}
\left(1+\frac{1}{5\log^3 y}\right)^2.
\end{equation}

We take $y=\sqrt{x/\log x}$.
We note that $y>e^{30}$ since we have assumed that $x>e^{73}$.
Hence, we observe that
\begin{equation}
\left(1+\frac{1}{5\log^3 y}\right)^2<\frac{\log x}{\log y}\left(1+\frac{0.49}{\log^3 x}\right)
\end{equation}
and \eqref{eq22} yields that
\begin{equation}
Z(A, y, \Omega)
<\frac{8e^\gamma x\left(1+\frac{1}{\log x}\right)\left(1+\frac{0.49}{\log^3 x}\right)}
{P_U(x)\log x\left(1-\frac{\log\log x-8\gamma}{\log x}\right)^2\left(1-\frac{\log\log x}{\log x}\right)}.
\end{equation}
Now Theorem \ref{th21} immediately follows from \eqref{eq21}.

\section{Proofs of Theorems \ref{th11} and \ref{th13}}

In this section, we prove Theorems \ref{th11} and \ref{th13}.
We put $U$ to be the set of prime factors $p$ of $N$ such that $p^2$ does not divide $N$,
so that $N_1=\prod_{p\in U} p$.
As we noted in the last section, $U$ must be $1$-self-repulsive if $M\vph^*(N)=N\pm 1$ and
$(-1)$-self-repulsive if $N=M\sigma^*(N)\pm 1$.

We begin by proving Theorem \ref{th11}.
Assume that $N$ is a positive integer satisfying $M\vph^*(N)=N\pm 1$ for some integer $M\geq 2$.
Let $x_1$ be the largest prime factor of $N_1$.
We note that $P_U(x_1)=\prod_{p\in U} p/(p-1)=N_1/\vph(N_1)$ and $\theta_U(x_1)=\sum_{p\in U}\log p\leq \log N_1$.

We begin by proving that $N_1/\vph(N_1)<15.68996\log\log\log N_1$.
Let $x_0=e^{73}$.
We discuss three cases:
(i) $x_1\leq x_0$,
(ii) $x_1>x_0$, $\theta_U(x_1)\geq x_1/\log\log x_1$, and
(iii) $x_1>x_0$, $\theta_U(x_1)<x_1/\log\log x_1$.
In the case (iii), we put $x_2$ be the largest number $x$ such that $\theta_U(x)\geq x/\log\log x$
and $x_3=\theta_U(x_1)$.
Then we settle four subcases.
(a) $x_3>x_2$ and $x_2\leq x_0$, (b) $x_3>x_2>x_0$,
(c) $x_3\leq x_2\leq x_0$, and (d) $x_3\leq x_2$ and $x_2>x_0$.

In the case (i), putting $p_1$ to be the largest prime such that $\prod_{p\leq p_1} p\leq N_1$,
the Corollary of Theorem 8 in \cite{RS} gives that
\begin{equation}\label{eq30}
\frac{N_1}{\vph(N_1)}\leq P(p_1)<\frac{e^\gamma}{2}\left(\log p_1+\frac{1}{\log p_1}\right)
<15.15486 \log\log p_1,
\end{equation}
where the last inequality follows from the fact that $p_1\leq x_1\leq x_0$ .
If $p_1>500000$, then Theorem 1 of \cite{BKLNW} gives that
$p_1<1.0268\theta(p_1)<1.0268\log N_1$ and we obtain $N_1/\vph(N_1)<15.56102\log\log\log N_1$,
which is more than we desired.
If $p_1<500000$ and $N_1>3704$, then we have $P(p_1)<11.68731<15.68996\log\log\log N_1$.
If $N_1=19$ or $23\leq N_1\leq 3703$, then we can confirm $N_1/\vph(N_1)<7.34789\log\log\log N_1$ by calculation.

Assume that $x_1>x_0$.
As we have seen in the last section, $U$ must be $1$-self-repulsive.
Let $x$ be a real number such that $x_0\leq x\leq x_1$ and $\theta_U(x)\geq x/\log\log x$.
Observing that $\pi_U(x)\geq\theta_U(x)/\log x>x/(\log x\log\log x)$, Theorem \ref{th21} immediately gives that
\begin{equation}\label{eq31}
P_U(x)<\frac{8e^\gamma\left(1+\frac{1}{\log x}\right)\left(1+\frac{1}{2\log^3 x}\right)}
{\left(1-\frac{\log\log x-8\gamma}{\log x}\right)^2\left(1-\frac{\log\log x}{\log x}\right)}\log\log x.
\end{equation}
Since we have assumed that $\theta_U(x)\geq x/\log\log x$,
we have $\log \theta_U(x)>\log x-\log\log\log x$
and therefore $\log\log \theta_U(x)>\log(\log x-\log\log\log x)
>\log\log x-1.01011\log\log\log x/\log x$.
Hence, \eqref{eq31} gives that
\begin{equation}\label{eq32}
P_U(x)<8e^\gamma \delta(\log x)\log\log\theta_U(x),
\end{equation}
where
\begin{equation}
\delta(t)=\frac{\left(1+\frac{1}{t}\right)\left(1+\frac{1}{2t^3}\right)}
{\left(1-\frac{\log t-8\gamma}{t}\right)^2\left(1-\frac{\log t}{t}\right)
\left(1-\frac{1.01011\log\log t}{t\log t}\right)}.
\end{equation}
We see that
\begin{equation}\label{eq33}
\log \delta(t)<\frac{1}{t}+\frac{1}{2t^3}+\frac{3\log t-16\gamma+1.01011(\log\log t)/\log t}{t}+\delta_1(t)
\end{equation}
for $t>73$, where
\begin{equation}\label{eq33a}
\begin{split}
\delta_1(t)= & ~ \frac{(\log t-8\gamma)^2}{t^2(1-\abs{(\log t-8\gamma)/t})}+\frac{\log^2 t}{2t^2(1-(\log t)/t)} \\
< & ~ \frac{1.06245(\log t-8\gamma)^2+0.53123\log^2 t}{t^2}<\frac{0.13552}{t}.
\end{split}
\end{equation}
We can easily see that $1/(2t^2)+1.01011(\log\log t)/\log t<0.34298$
and \eqref{eq33} implies that
\begin{equation}\label{eq33b}
\log\delta(t)<\frac{3\log t-16\gamma+1.34298+0.13552}{t}<\frac{3\log t-7.75695}{t}
\end{equation}
and, observing that $(3\log t-7.75695)/t<0.07007$ for $t>73$,
\begin{equation}\label{eq34}
\begin{split}
\delta(t)< & ~ 1+\frac{3\log t-7.75695}{t}+\frac{(3\log t-7.75695)^2}{2(1-0.07007)t^2} \\
< & 1+\frac{3\log t-7.55957}{t}.
\end{split}
\end{equation}

In the case (ii), taking $x=x_1$, we have
$P_U(x_1)=N_1/\vph^*(N_1)$ and $\theta_U(x_1)=\log N_1$ as we noted above.
Hence, \eqref{eq32} together with \eqref{eq34} yield that
\begin{equation}
\begin{split}
\frac{N_1}{\vph^*(N_1)}< & ~ 8e^\gamma\left(1+\frac{3\log\log x_1-7.55957}{\log x_1}\right)\log\log\log N_1 \\
< & ~ 15.28538\log\log\log N_1.
\end{split}
\end{equation}

Now we settle the remaining case (iii).
If $x_3\geq x_2$, then, partial summation gives
\begin{equation}
\begin{split}
& S_U(x_1)-S_U(x_2)=\frac{\theta_U(x_2)}{x_2\log x_2}-\frac{\theta_U(x_1)}{x_1\log x_1}
+\int_{x_2}^{x_1} \frac{\theta_U(t)(1+\log t)}{t^2 \log^2 t}dt \\
& ~ <\frac{1}{\log x_2\log\log x_2}+\int_{x_2}^{x_3} \frac{\theta_U(t)(1+\log t)}{t^2 \log^2 t}dt
+ x_3\int_{x_3}^{x_1} \frac{1+\log t}{t^2 \log^2 t}dt \\
& ~ <\log\log\log x_3-\log\log\log x_2+\frac{1}{\log x_2\log\log x_2}+\frac{1}{\log x_2},
\end{split}
\end{equation}
where we see that $\theta_U(t)\leq x_3$ for $t\leq x_1$.
Since
\begin{equation}
\log \frac{P_U(x_1)}{P_U(x_2)}<\sum_{x_3<p\leq x_1}\sum_{m=1}^\infty \frac{1}{mp^m}
<S_U(x_1)-S_U(x_2)+\frac{1}{2(x_0-1)},
\end{equation}
we have
\begin{equation}
\frac{P_U(x_1)}{P_U(x_2)}<\frac{\log\log x_3}{\log\log x_2}\exp\left(\frac{1.233076}{\log x_0}\right).
\end{equation}

Now, in the case (a), then, the Corollary of Theorem 8 in \cite{RS} gives that
$P_U(x_2)\leq \prod_{3\leq p\leq x_2}p/(p-1)<(e^\gamma/2)(\log x_2+1/\log x_2)$
and
\begin{equation}\label{eq35}
P_U(x_1)<\frac{e^\gamma \log\log x_3}{2\log\log x_2}\exp\left(\frac{1.233076}{\log x_0}\right)
\left(\log x_2+\frac{1}{\log x_2}\right)
\end{equation}
and $P_U(x_1)<15.41303\log\log\log N_1$, which is more than desired.
In the other case (b), then, taking $x=x_2$ in \eqref{eq32},
we have $P_U(x_2)<8e^\gamma \delta(\log x_2)\log\log \theta_U(x_2)$.
Since $\theta_U(x_2)<(1+10^{-10})x_2$ from Theorem 1 of \cite{BKLNW},
we obtain
\begin{equation}\label{eq36}
\begin{split}
P_U(x_1)< & ~ 8e^\gamma \delta(\log x_2)\left(1+\frac{1.23308}{\log x_2}\right)\log\log x_3 \\
< & ~ 15.54576\log\log x_3
\end{split}
\end{equation}
with the aid of \eqref{eq34}.
This immediately yields that $N_1/\vph(N_1)=P_U(x_1)<15.54576\log\log\log N_1$ as desired.

If $x_3<x_2$, then we have
\begin{equation}
\begin{split}
S_U(x_1)-S_U(x_2)< & ~ \frac{1}{\log x_2\log\log x_2}+ x_3\int_{x_2}^{x_1} \frac{1+\log t}{t^2 \log^2 t}dt \\
< & ~ \frac{1}{\log x_2\log\log x_2}+\frac{1}{\log x_2}.
\end{split}
\end{equation}
In the case (c), like above, we have
\begin{equation}
P_U(x_2)\leq P(x_2)<\frac{e^\gamma}{2}\left(\log x_2+\frac{1}{\log x_2}\right)
\end{equation}
and therefore
\begin{equation}\label{eq37}
P_U(x_1)<\frac{e^\gamma}{2}\left(\log x_2+\frac{1}{\log x_2}\right)\exp\left(\frac{1.233076}{\log x_2}\right).
\end{equation}
We observe that $\log N_1=\theta_U(x_1)\geq \theta_U(x_2)\geq x_2/\log\log x_2$
and therefore $N_1/\vph(N_1)=P_U(x_1)<15.63054\log\log\log N_1$.
In the case (d), taking $x=x_2$ in \eqref{eq32} and proceeding as above, we obtain
\begin{equation}\label{eq38}
P_U(x_1)<8e^\gamma \delta(\log x_2)\left(1+\frac{1.23308}{\log x_3}\right)\log\log x_2.
\end{equation}
Observing that $\log N_1\geq x_2/\log\log x_2$ with $x_2\geq x_0$ and using \eqref{eq34}, we have
$N_1/\vph(N_1)<15.76514\log\log\log N_1$.

Hence, we have $N_1/\vph(N_1)<15.76514\log\log\log N_1$ in any case and conclude that
\begin{equation}
M\leq\frac{N+1}{\vph^*(N)}\leq\frac{1}{N}+\frac{N_1}{\vph(N_1)}\prod_{p^2\mid N}\frac{p^2}{p^2-1}
<19.44947\log\log\log N_1.
\end{equation}
Moreover, if $M\vph(N)=N\pm 1$, then $N=N_1$ and therefore $M=(N\pm 1)/\vph(N)<15.76515\log\log\log N$,
which completes the proof of Theorem \ref{th11}.

We can prove Theorem \ref{th13} in a quite similar way with $x_0=e^{95}$ instead of $e^{73}$.
If $\sigma^*(N)=MN+1$, then $U$ must be $(-1)$-self-repulsive.
If $x_1\leq x_0$, then, putting $p_1$ to be the largest prime such that $\prod_{p\leq p_1} p\leq N_1$,
we have
\begin{equation}
\frac{\psi(N_1)}{N_1}\leq \prod_{p\leq p_1}\frac{p+1}{p}
<\frac{4 e^\gamma(1+1/p_1)}{\pi^2}\left(\log p_1+\frac{1}{\log p_1}\right)
\end{equation}
and, proceeding as above, we obtain $\psi(N_1)/N_1<15.76515\log\log\log N_1$, provided that $N_1\geq 19$.

Now assume that $x_1>x_0$.
Let $x$ be a real number such that $x_0\leq x\leq x_1$ and $\theta_U(x)\geq x/\log\log x$.
Proceeding as above, we obtain \eqref{eq32} and then \eqref{eq33} with $1.00807$ in place of $1.01023$
and, for $t>95$,
\begin{equation}
\begin{split}
\delta_1(t)<\frac{1.05035(\log t-8\gamma)^2+0.52518\log^2 t}{t^2}<\frac{0.11669}{t}
\end{split}
\end{equation}
instead of \eqref{eq33a}.
We can easily see that $1/(2t^2)+1.00807(\log\log t)/\log t<0.33564$
to obtain $\log\delta(t)<(3\log t-7.78512)/t<0.06186$ for $t>95$ instead of \eqref{eq33b}.
Hence, we obtain \eqref{eq34} with $7.59129$ replaced by $7.55957$.

In the case $\theta_U(x_1)\geq x_1/\log\log x_1$, like above, we have
\begin{equation}
\begin{split}
\frac{\psi(N_1)}{N_1}< & ~ 8e^\gamma\left(1+\frac{3\log\log N_1-7.59129}{\log N_1}\right)\log\log\log N_1 \\
< & ~ 15.15904\log\log\log N_1.
\end{split}
\end{equation}

If $x_3>x_2$ and $x_2\leq x_0$, then we have
\eqref{eq35} with $4(1+1/x_2)e^\gamma/\pi^2$ in place of $e^\gamma/2$,
which gives $P_U(x_1)<15.25485\log\log\log N_1$.
If $x_3>x_2>x_0$, then we have $P_U(x_1)<15.35491\log\log\log N_1$ instead of \eqref{eq36}.

If $x_3\leq x_2\leq x_0$, then 
we obtain \eqref{eq37} with $4(1+1/x_2)e^\gamma/\pi^2$ in place of $e^\gamma/2$
and then $P_U(x_1)<15.41935\log\log\log N_1$.
If $x_3\leq x_2$ and $x_2>x_0$, then we have \eqref{eq38}
and $P_U(x_1)<15.5205\log\log\log N_1$.

Hence, we have $\psi(N_1)/N_1<15.5205\log\log\log N_1$ in any case and therefore
\begin{equation}
M\leq\frac{\sigma^*(N)+1}{N}\leq\frac{1}{N}+\frac{\psi(N_1)}{N_1}\prod_{p^2\mid N}\frac{p^2+1}{p^2}
<18.87067\log\log\log N_1.
\end{equation}
Moreover, if $\psi(N)=MN\pm 1$, then $N=N_1$ and therefore $M=(\psi(N_1)\pm 1)/N_1<15.52051\log\log\log N$,
which completes the proof of Theorem \ref{th13}.

\section{Proofs of Theorems \ref{th12} and \ref{th14}}

We put $U$ to be the set of prime factors $p$ of $N$ such that $p^2$ does not divide $N$,
so that $N_1=\prod_{p\in U} p$.
Moreover, we put $r=\omega(N_1)\geq 4$.

We begin by proving Theorem \ref{th12}.
Our argument is similar to the proof of Theorem \ref{th11} in the last section but needs some modification.
Let $x_0=e^{72}$.
We discuss three cases:
(i) $x_1\leq x_0$,
(ii) $x_1>x_0$, $\pi_U(x_1)\geq x_1/(\log x_1 \log\log x_1)$, and
(iii) $x_1>x_0$, $\pi_U(x_1)>x_1/(\log x_1 \log\log x_1)$.
Moreover, in the case (iii), we put $x_2$ be the largest number $x$
such that $\pi_U(x)\geq x/(\log x\log\log x)$
and settle four subcases.
(a) $r\log r>x_2$ and $x_2\leq x_0$, (b) $r\log r>x_2>x_0$,
(c) $r\log r\leq x_2\leq x_0$, and (d) $r\log r\leq x_2$ and $x_2>x_0$.

In the case (i), then, putting $p_2$ to be the $r$-th odd prime, we have
\begin{equation}\label{eq40}
\frac{N_1}{\vph(N_1)}\leq P(p_2)<\frac{e^\gamma}{2}\left(\log p_2+\frac{1}{\log p_2}\right).
\end{equation}
where we use the fact that $p_2\leq x_1\leq x_0$.
If $4\leq r<e^{16}$, then we have $N_1/\vph(N_1)\leq P(p_2)<7.366803\log\log r$, which is more than we have desired.
If $r>e^{16}$, then we can derive from Proposition 5.15 of \cite{Dus} that
$\log\log p_2<1.059704\log\log r$ and \eqref{eq40} yields that
$N_1/\vph(N_1)<15.89085\log\log r$.

Now assume that $x_1>x_0$.
Let $x$ be a real number such that $x_0\leq x\leq x_1$ and $\pi_U(x)\geq x/(\log x \log\log x)$.
We observe that
\begin{equation}
\begin{split}
\log\log \pi_U(x)> & ~ \log(\log x-\log\log x-\log\log\log x) \\
> & ~ \log\log x-\frac{1.04204(\log\log x+\log\log\log x)}{\log x}.
\end{split}
\end{equation}
As in the last section, we obtain instead of \eqref{eq32},
\begin{equation}\label{eq41}
P_U(x)<8e^\gamma \eta(\log x)\log\log\pi_U(x),
\end{equation}
where $\eta(t)$ is defined as $\delta(t)$ with $1.04204(\log t+\log\log t)$ in place of \\
$1.01011\log\log t$.
Proceeding as in the last section, with $1.06315$, $0.53158$, $1.04214$, and $0.13736$
in place of $1.06245$, $0.53123$, $0.34298$, and $0.13622$ respectively, we have
$\log\eta(t)<(3\log t-7.05655)/t$
and, observing that $(3\log t-7.05655)/t<0.08019$, $\eta(t)<(3\log t-6.80452)/t$.

Hence, in the case (ii), we have
\begin{equation}
\frac{N_1}{\vph^*(N_1)}<8e^\gamma\left(1+\frac{3\log\log x_1-6.80452}{\log x_1}\right)\log\log\pi_U(x_1)
\end{equation}
and, since $\pi_U(x_1)=\omega(N_1)$, we conclude that
\begin{equation}
\frac{N_1}{\vph^*(N_1)}<15.44101\log\log\omega(N_1).
\end{equation}

Now we settle the remaining case (iii).

If $r\log r>x_2$, then, partial summation gives
\begin{equation}
\begin{split}
& S_U(x_1)-S_U(x_2)=\frac{\pi_U(x_2)}{x_2}-\frac{\pi_U(x_1)}{x_1}
+\int_{x_2}^{x_1} \frac{\pi_U(t)}{t^2}dt \\
& ~ <\frac{1}{\log x_2\log\log x_2}+\int_{x_2}^{r\log r} \frac{\pi_U(t)}{t^2}dt
+ r\int_{r\log r}^{x_1} \frac{dt}{t^2} \\
& ~ <\log\log\log (r\log r)-\log\log\log x_2+\frac{1}{\log x_2\log\log x_2}+\frac{1}{\log r}
\end{split}
\end{equation}
and
\begin{equation}
\frac{P_U(x_1)}{P_U(x_2)}<\frac{\log\log (r\log r)}{\log\log x_2}(1+\epsilon(x_2, r)),
\end{equation}
where
\begin{equation}
1+\epsilon(x_2, r)=\exp\left(\frac{1}{\log x_2\log\log x_2}+\frac{1}{\log r}+\frac{1}{2(x_2-1)}\right).
\end{equation}

Like above, in the case (a), observing that
$P_U(x_2)<(e^\gamma /2)(\log x_2+1/\log x_2)$ we have
\begin{equation}\label{eq42}
\begin{split}
P_U(x_1)< \frac{e^\gamma\log\log (r\log r)}{2\log\log x_2}\left(\log x_2+\frac{1}{\log x_2}\right)
(1+\epsilon(x_2, r))
\end{split}
\end{equation}
and therefore $N_1/\vph(N_1)=P_U(x_1)<15.48645\log\log r$.
In the case (b), we have
\begin{equation}\label{eq43}
\begin{split}
P_U(x_1)<8e^\gamma \eta(\log x_2)(1+\epsilon(x_2, r))\log(\log r+\log\log r)
\end{split}
\end{equation}
and $N_1/\vph(N_1)<15.94648 \log\log r$.

If $r\log r\leq x_2$, then we have
\begin{equation}
\begin{split}
S_U(x_1)-S_U(x_2)< & ~ \frac{1}{\log x_2\log\log x_2}+ r\int_{x_2}^{x_1} \frac{dt}{t^2}dt \\
< & ~ \frac{1}{\log x_2\log\log x_2}+\frac{1}{\log r}.
\end{split}
\end{equation}

In the case (c), then, as in the last section, we have
\begin{equation}\label{eq44}
P_U(x_1)<\frac{e^\gamma}{2}(1+\epsilon(x_2, r))\left(\log x_2+\frac{1}{\log x_2}\right).
\end{equation}
Observing that $r\geq \pi_U(x_2)\geq x_2/(\log x_2 \log\log x_2)$,
we have
$N_1/\vph(N_1)<15.56984 \log\log r$.
In the case (d), taking $x=x_2$ in \eqref{eq41} and proceeding as above, we obtain
\begin{equation}\label{eq45}
P_U(x_1)<8e^\gamma \eta(\log x_2)(1+\epsilon(x_2, r))\log\log x_2.
\end{equation}
We observe that $r\geq \pi_U(x_2)\geq x_2/(\log x_2 \log\log x_2)$ with $x_2>x_0$,
to obtain $N_1/\vph(N_1)<16.03234 \log\log r$.

Hence, we obtain $N_1/\vph(N_1)<16.03234\log\log r$ in any case.
As in the last section, we have $M\leq (N+1)/\vph^*(N)<19.77911\log\log r$
and, if $M\vph(N)=N\pm 1$, $M<16.03235\log\log r$, proving Theorem \ref{th12}.

Now we prove Theorem \ref{th14}.
If $x_1\leq e^{93}$, then, putting $p_2$ to be the $r$-th odd prime, we have
\begin{equation}
\frac{\psi(N_1)}{N_1}<\frac{4e^\gamma(1+1/p_2)}{\pi^2}\left(\log p_2+\frac{1}{\log p_2}\right)
\end{equation}
instead of \eqref{eq40} and $\psi(N_1)/N_1<15.69684\log\log r$.

Now assume that $x_1>e^{93}$.
Let $x$ be a real number such that $e^{93}\leq x\leq x_1$ and $\pi_U(x)\geq x/(\log x \log\log x)$.
Like above but replacing $1.04204$, $1.06315$, $0.53158$, $1.04214$, $0.13736$, $-7.05655$, and $0.08019$
by $1.03398$, $1.05124$, \\ $0.52562$, $1.03404$, $0.1162$, $-7.08521$, and $0.07003$,
we obtain \eqref{eq41} with $\eta(t)<(3\log t-6.8383)/t$.

If $\pi_U(x_1)>x_1/(\log x\log\log x)$, then, proceeding as above,
we have $P_U(x_1)<15.28421\log\log r$.

If $r\log r\geq x_2$ and $x_2\leq e^{93}$, then we have
\eqref{eq42} with $4(1+1/x_2)e^\gamma/\pi^2$ in place of $e^\gamma/2$,
which gives $P_U(x_1)<15.18184\log\log r$.
If $r\log r\geq x_2>e^{93}$, then we have $P_U(x_1)<15.66533\log\log r$ instead of \eqref{eq43}.

If $r\log r<x_2\leq e^{93}$, then 
we obtain \eqref{eq44} with $4(1+1/x_2)e^\gamma/\pi^2$ in place of $e^\gamma/2$
and then $P_U(x_1)<15.24232\log\log r$.
If $r\log r<x_2$ and $x_2\geq e^{93}$, then we have \eqref{eq45}
and $P_U(x_1)<15.72774\log\log r$.

Now we have confirmed that $\psi(N_1)/N_1<15.72774\log\log r$ in any case.
Hence, we obtain $M\leq (\sigma^*(N)+1)/N<19.40333\log\log r$
and, if $\psi(N)=MN\pm 1$, $M<15.72775\log\log r$.
This completes the proof.

{}
\end{document}